\documentclass[twoside]{amsart} 

\usepackage{pinlabel}
\usepackage{amsmath, amsfonts, amsthm, amssymb}
\usepackage{exscale}
\usepackage{cite}
\usepackage{amscd}
\usepackage{graphics}
\usepackage{color}
\usepackage{dsfont}
\usepackage{graphicx}

\usepackage[pdftex]{hyperref}
\hypersetup{
    colorlinks=true,       
    linkcolor=red,          
    citecolor=blue,        
    filecolor=magenta,      
    urlcolor=cyan           
}
\usepackage[utf8]{inputenc}

\renewcommand{\geq}{\geqslant}
\renewcommand{\leq}{\leqslant}

\newtheorem{theo}{Theorem}
\newtheorem*{theo*}{Theorem}
\newtheorem{prop}{Proposition}[section]

\newtheorem{coro}[prop]{Corollary}
\newtheorem{lemma}[prop]{Lemma}
\newtheorem*{main-theorem}{Main Theorem}
\newtheorem*{theorem*}{Theorem}

\theoremstyle{definition}

\newtheorem{rema}[prop]{Remark}
\newtheorem{definition}[prop]{Definition}
\newtheorem*{remark*}{Remark}

\numberwithin{equation}{section}

\def\phi{\varphi}
\def\Re{\,\mathrm{Re}\,}
\def\Im{\,\mathrm{Im}\,}

\def\div{\mathrm{div}\,}

\def\Op{\mathrm{Op}\,}

\def\phi{\varphi}

\def\be{\begin{eqnarray*}}
\def\ee{\end{eqnarray*}}
\def\ben{\begin{eqnarray}}
\def\een{\end{eqnarray}}

\def\L2R{L_{\text{Rest}}^2}

\def\11{\mathds{1}}

\def\L2c{L^2_{\text{comp}}}

\begin{document}

\title[Kelvin-Voigt damping]{Decay rates for Kelvin-Voigt damped wave equations II: the geometric control condition}

\author[N. Burq]{Nicolas Burq}
\address{Universit\'e Paris-Saclay, Laboratoire de mathématiques d'Orsay, UMR 8628 du CNRS, B\^atiment 307, 91405 Orsay Cedex, France and Institut Universitaire de France}
\email{nicolas.burq@math.u-psud.fr}
\author[C-M. Sun]{Chenmin Sun}
\address{Universit\'e de Cergy-Pontoise, Laboratoire de Mathématiques AGM, UMR  8088 du CNRS, 2 av. Adolphe Chauvin
95302 Cergy-Pontoise Cedex, France }
\email{chenmin.sun@u-cergy.fr}


\begin{abstract}
 We study in this article decay rates for  Kelvin-Voigt damped wave equations under a geometric control condition. When the damping coefficient is sufficiently smooth ($C^1$ vanishing nicely, see~\eqref{Hypo1}) we show that exponential decay follows from geometric control conditions (see~\cite{BuCh, Te12} for similar results under stronger assumptions on the damping function).
\end{abstract}

\maketitle

\section{Introduction}
\label{S:intro}
In  this paper we investigate decay rates for Kelvin-Voigt damped wave equations under  geometric control conditions. 
  We work in a smooth bounded domain $\Omega \subset \mathbb{R}^d$ and consider the following equation 
\begin{equation}\label{damped}
\left\{ \begin{aligned} &(\partial_t ^2 - \Delta) u - \text{ div}( a(x) \nabla_x  \partial_t u) =0 \\
& u \mid_{t=0} = u_0 \in H_0^1( \Omega),  \quad \partial_t u \mid_{t=0} = u_1 \in L^2 ( \Omega)\\
&u \mid_{\partial \Omega}  =0 \end{aligned}
\right. 
\end{equation}
with a non negative damping term $a(x)$. The solution can be written as 
\begin{equation}
 U(t) = \begin{pmatrix} u \\ \partial_t u \end{pmatrix} = e^{\mathcal{A} t } \begin{pmatrix} u_0 \\ u_1 \end{pmatrix},
 \end{equation}
 where the generator $\mathcal{A}$ of the semi-group is given by 
 $$ \mathcal{A} = \begin{pmatrix} 0 & 1 \\ \Delta &  \text{ div} a \nabla \end{pmatrix}  \begin{pmatrix} u_0 \\ u_1 \end{pmatrix} ,$$
 with domain 
 $$ {D}( \mathcal{A}) = \{  (u_0, u_1) \in H^1_0\times L^2;  \Delta u_0 +\text{ div} a \nabla u_1 \in L^2; u_1 \in H^1_0\}.$$ 
The energy of solutions 
$$E(u) (t) = \int_{\Omega} (|\nabla_x u | ^2 + |\partial_t u| ^2) dx $$
satisfies 
$$ E((u_0, u_1)) (t) - E((u_0, u_1)) (0) = - \int_{0}^t\int_{\Omega} a(x) |\nabla_x \partial_t u |^2 (s,x) ds.$$
 Our purpose here is to show that if the damping $a$ is sufficiently smooth, the exponential decay rate holds, dropping some unnecessary assumptions on the behaviour of the damping term where it becomes positive in  previous works~\cite{BuCh}.  Namely we shall assume
 $a(x)\geq 0$ is  $C^1(\overline{\Omega})$ and satisfy the regularity hypothesis
\begin{align}\label{Hypo1}
|\nabla a|\leq Ca^{\frac{1}{2}}.
\end{align}

Our main result is 

\begin{theo}\label{thm1}
Assume that $\Omega$ is a compact Riemannian manifold with smooth boundary. Let $a\in C^1(\overline{\Omega})$ be a nonnegative function satisfying \eqref{Hypo1}, such that the following geometric control condition is satisfied:
\begin{itemize}
	\item There exists  $\delta>0$ such that all rays of geometric optics (straight lines) reflecting on the boundary according to the
laws of geometric optics eventually reach the set $\omega_{\delta}=\{x\in\Omega:a(x)>\delta \}$ in finite time.
\end{itemize} 
 Then there exists $\alpha>0$, such that for all $t\geq 0$ and every $(u_0,u_1)\in H_0^1(\Omega)\times L^2(\Omega)$, the energy of solution $u(t)$ of \eqref{damped} with initial data $(u_0,u_1)$ satisfies
$$ E[u](t)\leq e^{-\alpha t}E[u](0).
$$ 
\end{theo}

To prove this result, we first reduce it very classicaly in Section~\ref{sec.2} to resolvent estimates. Since the low frequency estimates are true, we are reduced to the high frequency regime. The proof relies on resolvent estimates which are proved through a contradiction argument that we establish in Section~\ref{sec.2}. In Section~\ref{sec.3} we prove  {\em a priori} estimates for our sequences. The main task then is to prove a propagation invariance for these measures. A main difficulty to overcome is that it is not possible to put the damping term in the r.h.s. of the equation~\eqref{damped} and treat it as a {\em perturbation}. Instead we have to keep it on the left hand side and revisit the proof of the propagation property from~\cite{BuLe01}.  In Section~\ref{sec.4}, we  introduce the geometric tools necessary to tackle the boundary value problem and define semi-classical measures associated to our sequences.  In Section~\ref{sec.5} we prove the interior propagation result for our measures. Finally, in Section~\ref{sec.7}, we finish the proof of the contradiction argument by establishing the invariance of the semi-classical measures we defined up to the boundary. Here the proof uses crucially the main result in~\cite[Théorème 1]{BuLe01}.
\begin{rema}\label{rem.fredholm}
Throughout this note, we shall prove that some operators of the type $P- \lambda \text{Id}$, $\lambda \in \mathbb{R}$ (resp.  $\lambda \in i \mathbb{R}$) are invertible with estimates on the inverse. All these operators share the feature that they have compact resolvent, i.e. $\exists z_0 \in \mathbb{C}; (P- z_0)^{-1}$  exists and is compact (or it will be possible to reduce the question to this situation). 
As a consequence, since 
$$ (P- \lambda) = (P-z_0)^{-1} ( \text{Id} + (z_0- \lambda) )^{-1}), $$
and  $( \text{Id} + (z_0- \lambda)^{-1})$ is Fredholm with index $0$,  to show that $(P- \lambda) $ is invertible with inverse bounded in norm by $A$ , it is enough to bound the solutions of $(P- \lambda) u =f$ and prove 
$$ (P- \lambda) u =f\Rightarrow \| u \|_{L^2} \leq A \| f\|_{L^2} .$$
\end{rema}
\begin{rema}\label{C2case}
Assume that $a$ is the restriction to $\Omega$ of a nonnegative $C^2(\mathbb{R}^d)$ function. Then  the hypothesis~\eqref{Hypo1} is satisfied.
\end{rema}
\begin{proof}	
It is enough to prove~\eqref{Hypo1} for $\Omega = \mathbb{R}^d$,  $a\in  C^2( \Omega)$. 
Let $x_0 \in \mathbb{R}^d$ and denote by $z_0=\nabla a(x_0)$ From Taylor's formula, we have for any $s\in \mathbb{R}$, there exists $\theta \in(0,1)$, such that
$$ a(x_0+sz_0)=a(x_0)+s|z_0|^2+\frac{s^2}{2}a''(x_0+\theta sz_0) ( z_0, z_0)\geq 0
$$
Since this polynomial in $s$ is non negative, we deduce tat its discriminant is non positive
$$ |z_0|^4 - 2 \| a''\|_{\infty} |z_0 |^2 a(z_0)\leq 0 \Rightarrow |\nabla_x a (x_0)||^2 \leq 2 \| a''\|_{\infty} a(z_0).
$$
Notice that in the above lemma, the condition {\em cannot} be relaxed to $a \in C^2( \overline {\Omega}), a\geq 0$.  Indeed, consider the following example: $\Omega=B(0,1)$ and $a(x)=1-|x|^2$ for $|x|\leq 1$. Then obviously $a\in C^2(\overline{\Omega})$, $a\geq 0$ , but on the boundary, $\nabla_x a \neq 0$, while $a=0$.\end{proof}

\subsection*{Acknowledgment}
The first author is supported by Institut Universitaire de France and
ANR grant ISDEEC, ANR-16-CE40-0013. The second author is supported by the postdoc programe: ``Initiative d'Excellence Paris Seine" of CY Cergy-Paris Universit\'e and ANR grant
ODA (ANR-18-CE40- 0020-01).

\section{Contradiction argument}\label{sec.2}

It is well known that decay estimates for the evolution semi-group follow from resolvent estimates~\cite{Bu98, BoTo10, BaDu08}. Here we shall need 
the classical (see e.g. \cite[Proposition A.1]{BuGe18})
\begin{theo}
The exponential decay of the Kelvin Voigt semi-group is equivalent to the following resolvent estimate: There exists $C$ such that for all $\lambda \in \mathbb{R}$, the operator $(\mathcal{A} - i \lambda) $ is invertible from $D(\mathcal{A})$ to $\mathcal{H}$ and its inverse satisfies
\begin{equation}\label{Resolv.1}
\| ( \mathcal{A}- i \lambda) ^{-1}\| _{\mathcal{L} ( \mathcal{H})} \leq C 
\end{equation}\end{theo}
Let us first recall that 
 \begin{equation}\label{above} (\mathcal{A} - i \lambda)\begin{pmatrix} u\\v\end{pmatrix} = \begin{pmatrix}f\\ g\end{pmatrix} \Leftrightarrow \left\{ \begin{aligned} &-i \lambda u + v = f \\
&\Delta u + \text{div} a \nabla_x v- i \lambda v= g \end{aligned}\right.
\end{equation}
From~\cite[Section 2]{Bu20}, we have the following low frequencies  estimates of the resolvent of the operator $\mathcal{A}$: 
\begin{prop} \label{resolv.3}Assume that $a\in L^\infty$ is non negative  $a\geq 0$ and non trivial $\int_\Omega a(x) dx >0$.
Then for any $M>0$, there exists $C>0$ such that for all $\lambda \in \mathbb{R}, |\lambda | \leq M$, the operator $\mathcal{A} - i \lambda$ is invertible from $D( \mathcal{A})$ to $\mathcal{H}$ with estimate
\begin{equation}\label{resolv}
\| ( \mathcal{A} - i \lambda)^{-1} \|_{\mathcal{L}( \mathcal{H})} \leq C .
\end{equation}
\end{prop}
As a consequence, to prove Theorem~\ref{thm1} it is enough to study the high frequency regime $\lambda \rightarrow + \infty$ and prove 
 \begin{prop}\label{resolvent-HF}
Assume that $a\in C^1(\overline{\Omega})$ is a nonnegative function satisfying \eqref{Hypo1}. Then under the geometric control condition,  there exists $\Lambda_0 >0$ such that for any $|\lambda|> \Lambda_0$ we have
$$ \|(\mathcal{A}-i\lambda)^{-1}\|_{\mathcal{L}(\mathcal{H})}\leq C.
$$ 
\end{prop}

By standard argument, we can reduce the proof of Proposition \ref{resolvent-HF} to a semi-classical estimate. We denote by $0<h=|\lambda|^{-1}\ll 1$ and
$$ P_h=-h^2\Delta-1-ih\div a(x)\nabla.
$$
\begin{prop}\label{resolvent-HFsemi}
There exists $C>0$, such that for all $0<h\ll 1$,
\begin{align}\label{eq:resolvent-HFsemi}
 \|P_h^{-1}\|_{\mathcal{L}(L^2)}\leq Ch^{-1}.
\end{align}
\end{prop}
For the proof of Proposition \ref{resolvent-HFsemi}, we argue by contradiction. Assume that there exist sequences $(u_n)\subset H^2\cap H_0^1, (f_n)\subset L^2$ and $h_n\rightarrow 0$, such that $P_{h_n}u_n=f_n$, $\|u_n\|_{L^2}=1$ and $\|f_n\|_{L^2}=o(h_n)$. We will use a semi-classical notation and denote by $(u_h,f_h)$ the sequences with the properties
\begin{align}\label{2.5}
\|u_h\|_{L^2}=1, \quad \|f_h\|_{L^2}=o(h), \quad P_hu_h=f_h.
\end{align}
Sometimes we even omit the subindex for $u_h, f_h$. In the following subsections, we will prove  propagation estimates for such sequences.
\section{A priori estimates}\label{sec.3}
In this section we establish a series of {\em a priori} estimates for the sequence defined in~\eqref{2.5}. 
\begin{lemma}\label{eenergy}
Assume that $a\in L^{\infty}(\Omega)$ is a non-negative function, then
\begin{align}
& (1) \quad \int_{\Omega}\big(|u_h|^2-|h\nabla u_h|^2\big) =\Re\int_{\Omega}f_h\overline{u_h}=o(h);\label{E1}\\
& (2)\quad \int_{\Omega}a(x)|h\nabla u_h|^2=h\Im\int_{\Omega}f_h\overline{u_h}=o(h^2). \label{E2};\\
& (3) \quad h^2\|u_h\|_{H^2}=O(1).
\end{align}	
\end{lemma}
\begin{proof}
	We get (1) and (2) by multiplying the equation $P_hu=f$ by $\overline{u}$, integrating by part and taking the real  and imaginary parts repectively. For (3), from the equation, we have
	$$ h^2\Delta u+u+ih\nabla a\cdot\nabla u+iha\Delta u=-f,\text{ i.e. } h^2\Delta u=-\frac{u+f+i\nabla a\cdot h\nabla u}{1+ih^{-1}a(x)}.
	$$
	From the global estimate of the Poisson equation
	$$ \|w\|_{H^2}\leq C\|\Delta w\|_{L^2},\forall w\in H^2\cap H_0^1,
	$$
	we obtain that $\|h^2u\|_{H^2}=O(1)$.
\end{proof}
\begin{coro}\label{apriori}
	Assume that $a\in C^{1}(\overline{\Omega})$ is a non-negative function satisfying \eqref{Hypo1}, then
\begin{align}\label{eq:apriori}
\|a^{\frac{1}{2}}u_h\|_{L^2}+\|a^{\frac{1}{2}}h\nabla u_h\|+\|a^{\frac{1}{2}}h^2\Delta u_h\|_{L^2}=o(h).
\end{align}
\end{coro}
\begin{proof}
We only need to estimate $\int_{\Omega}a(x)|u|^2$, since $\int_{\Omega}a(x)|h\nabla u|^2=o(h^2)$ is just \eqref{E2}.
Multiplying $P_hu=f$ by $a\overline{u}$ and taking the real part, we have
\begin{equation*}
\int_{\Omega}a(x)|u|^2=\Re\int_{\Omega}h\nabla u\cdot h\nabla (a\overline{u})-\Im h\int_{\Omega}a(x)\nabla u\cdot \nabla(a(x)\overline{u})+\Im\int_{\Omega}a(x)f\overline{u}.
\end{equation*}
Since $|\nabla a|\leq Ca^{\frac{1}{2}}$, the first term on the r.h.s. can be bounded by 
$$ \|a^{\frac{1}{2}}h\nabla u\|_{L^2}^2+h\|\nabla a h\nabla u\|_{L^2}\|u\|_{L^2}=o(h^2).
$$ The third term of r.h.s is bounded by $o(h)\|a^{\frac{1}{2}}u\|_{L^2}$, and the second term can be bounded by
$$ h\Big|\int_{\Omega}a\nabla a \cdot \overline{u}\nabla u \Big|\leq h\|a^{\frac{1}{2}}\overline{u}\|_{L^2}\|a^{\frac{1}{2}}\nabla a\nabla u\|_{L^2}\leq C\|a^{\frac{1}{2}}h\nabla u\|\|a^{\frac{1}{2}}u\|_{L^2}=o(h)\|a^{\frac{1}{2}}u\|_{L^2}.
$$
For the second derivative, we observe that
$$ a(x)^{\frac{1}{2}}h^2\Delta u=-\frac{a^{\frac{1}{2}}u+a^{\frac{1}{2}}f+iha^{\frac{1}{2}}\nabla a\cdot h\nabla u}{1+ih^{-1}a(x)},
$$
thus $\|a^{\frac{1}{2}}h^2\Delta u\|_{L^2}=o(h)$.
This completes the proof of Corollary \ref{apriori}.
\end{proof}
Let $\nu$ be the out-normal vector field on $\partial\Omega$. We denote by $L^2(\partial)=L^2(\partial\Omega)$. The following hidden regularity holds:
\begin{lemma}\label{cachee}
Assume that $a\in C^1(\overline{\Omega})$ is a nonnegative function satisfying \eqref{Hypo1}, then
$$ \|h\partial_{\nu}u\|_{L^2(\partial)}=O(1),\quad \|a^{\frac{1}{2}}h\partial_{\nu}u\|_{L^2(\partial)}=O(h^{\frac{1}{2}}).
$$
\end{lemma}
\begin{proof}
We use the standard multiplier method. Let $L=b_j(x)\partial_j$ be an $C^2$ extension of the out-normal vector field $\nu$, where $b_j$'s are supported in a neighborhood of $\partial\Omega$. Write $P_h=P_{h,0}+iM_{h}$, where
$$ P_{h,0}=-h^2\Delta-1,\quad M_{h}=-h\div a(x)\nabla
$$
are self-adjoint operators. Consider the commutator $[P_h,L]=[P_{h,0},L]+i[M_{h},L]$. Note that $[P_{h,0},L]=\frac{1}{h}[P_{h,0},hL]$ belongs to $h^2\Op(S^2)$\footnote{Strictly speaking, the symbol of $L$ is not $C^{\infty}$, but here we only need  $\frac{1}{h}[P_{h,0},hL]=O(h^2)$ on $L^2$.  }, we deduce that $\big([P_{h,0},L]u,u\big)_{L^2}=O(1)$. By direct computation, we have
\begin{equation*}
\begin{split}
-\big([M_{h},L]u,u\big)_{L^2}=&\big(h\partial_k[a\partial_k,b_j\partial_j]u,u  \big)_{L^2}+h\big([\partial_k,b_j\partial_j]a\partial_ku,u \big)_{L^2}\\
=&\big(h\partial_k(a (\partial_kb_j)\partial_ju-b_j(\partial_j a)\partial_ku ),u \big)_{L^2}+\big((\partial_kb_j) h\partial_j(a\partial_ku),u \big)_{L^2}\\
=&-\big((a\partial_kb_j)\partial_ju-b_j(\partial_ja)\partial_ku,h\partial_ku \big)_{L^2}-\big(a\partial_ku, h\partial_j((\partial_kb_j) u). \big)_{L^2}
\end{split}
\end{equation*} 
From Corollary \ref{apriori}, the absolute value of the r.h.s. can be bounded by constant times
$$ \|a\nabla u\|_{L^2}+\|\nabla a\nabla u\|_{L^2}=o(1).
$$
Therefore, $\big([P_h,L]u,u\big)_{L^2}=O(1)$. On the other hand, by developing the commutator and exploiting the equation, we have
\begin{equation*}
\begin{split}
\big([P_h,L]u,u\big)_{L^2}=&\big(P_hLu,u\big)_{L^2}-\big(Lf,u\big)_{L^2}\\
=&\big(Lu,P_h^*u\big)_{L^2}-\big(f,L^*u\big)_{L^2}+h^2\|\partial_{\nu}u\|_{L^2(\partial)}^2+ih\|a^{\frac{1}{2}}\partial_{\nu}u\|_{L^2(\partial)}^2.
\end{split}
\end{equation*}
Observe that
$$\big(Lu,P_h^*u\big)_{L^2}=\big(Lu,f-2M_{h}u\big)_{L^2}=o(1)-2\big(Lu,M_{h} u\big)_{L^2}=o(1)-2\big(Lu,h\nabla a\cdot \nabla u+ha\Delta u\big)_{L^2}.
$$
Since $L$ is a first order differential operator and from Corollary \ref{apriori} that $\|a^{\frac{1}{2}}h\Delta u\|_{L^2}=o(1)$, we have 
$$ |\big(Lu,h\nabla a\cdot\nabla u+ha\Delta u \big)_{L^2}|\leq \|h\nabla u\|_{L^2}\|\nabla a \nabla u\|_{L^2}+\|a^{\frac{1}{2}}\nabla u\|_{L^2}\|a^{\frac{1}{2}}h\Delta u\|_{L^2}=o(1).
$$
Therefore,
$$ \|h\partial_{\nu}u\|_{L^2(\partial)}^2+ih\|a^{\frac{1}{2}}\partial_{\nu}u\|_{L^2(\partial)}^2=O(1).
$$
The proof of Lemma \ref{cachee} is then completed by taking real and imaginary parts.
\end{proof}

Let $\chi\in C_c^{\infty}(\mathbb{R})$ such that $\chi(z)\equiv 1$ for $|z|\leq 1$ and $\chi(z)\equiv 0$ for $|z|>2$. We decompose
\begin{equation}\label{decomposition}
u_h=v_h+w_h,\quad v_h=\chi\big(a h^{-1}\big)u_h,\quad  w_h=\big(1-\chi\big(a h^{-1}\big) \big)u_h.
\end{equation}
In the rest of this note, we always assume that $a\in C^1(\overline{\Omega})$ is a nonnegative function satisfying \eqref{Hypo1}.
\begin{lemma}\label{dampedregion}
We have
$$ \|w_h\|_{L^2}+\|h\nabla w_h\|_{L^2}=o(h^{\frac{1}{2}}),\quad	
 \|a^{\frac{1}{2}}v_h\|_{L^2}+ \|a^{\frac{1}{2}}h\nabla v_h\|_{L^2}=o(h),
$$
and
$$ \|u_h\|_{H_h^1(a\geq c h) }+\|v_h\|_{H_h^1(a\geq ch)}=o(h^{\frac{1}{2}}).
$$
\end{lemma}
\begin{proof}
	By definition,
$$ \int_{\Omega}\big(|w|^2+|h\nabla w|^2\big)\leq\int_{a\geq h}\big(|u|^2+|h\nabla u|^2+|\nabla a|^2|u|^2 \big).
$$
The conclusion then follows from Corollary \ref{apriori} and the fact that $|\nabla a|^2\leq Ca$. Similarly, for any other cutoff to the region $a\geq c h$, we deduce that
$ \|u_h\|_{H_h^1(a\geq c h)}=o(h^{\frac{1}{2}}).
$
For the estimate of $v$, note that
$a^{\frac{1}{2}}h\nabla v=a^{\frac{1}{2}}\chi h\nabla u+ a^{\frac{1}{2}}\nabla a\chi'v$, from Corollary \ref{apriori}, we have
$$ \|a^{\frac{1}{2}}h\nabla v\|_{L^2}\leq \|a^{\frac{1}{2}}h\nabla u\|_{L^2}+\|\chi'a^{\frac{1}{2}} (a^{\frac{1}{2}}v)\|_{L^2}=o(h).
$$
This completes the proof of Lemma \ref{dampedregion}.
\end{proof}

\section{Geometry, semi-classical measures}\label{sec.4}
Having the {\em a priori} estimates of the previous section at hand, we can now  study $v_h$. For some subsequence of $v_h$, we will associate it a semi-classical measure and then prove the invariance of the measure under the generalized geodesic flow. First recall some geometric preliminaries from \cite{BuLe01}.

 \subsection{Geometry}
Denote  by ${^b}T \Omega$  the bundle of rank $d$ whose sections are the vector fields tangent to $\partial \Omega$, $^bT^* \Omega$ the dual bundle (Melrose's compressed cotangent bundle) and  $j : T^* \Omega \rightarrow {^b}T^* \Omega$ the canonical map. In any coordinate system where $\Omega=\{ x= (x_{d}>0, x')\}$), the bundle $^bT \Omega$ is generated by the fields $\frac{\partial }{ \partial x'}$,
$x_{d} \frac{\partial }{ \partial x_{d}}$ and $j$ is defined by
\begin{equation} 
  j(x_{d},x',\xi_{d} ,\xi' )=(x_{d},x',v=x_{d}\xi_{d}, \xi' ) .
 \end{equation}
Denote by $\text{Car} {P_0}$ the semi-classical characteristic manifold of ${P}_0= -h^2\Delta -1$ and $Z$ its projection
\begin{equation} \label{eq:5}
\text{Car} {P}_0=\left\{ (x,\xi )= (x',x_d,\xi',\xi_d)\in T^* {\mathbb R}^d\mid _{\overline \Omega} ; p(x, \xi) =0\right\}, \qquad  
Z=j(\text{Car}P_0).
 \end{equation}
The set $Z$ is a locally compact metric space.\par
Consider, near a point $x_{0}\in \partial \Omega$ a geodesic system of coordinates for which $x_{0}= (0,0)$, $\Omega= \{(x_{d},x')\in {\mathbb R}^+\times {\mathbb R}^{ d-1}\}$ and the operator ${P}_0$ has the form  (near $x_{0}$)
\begin{equation}
\label{3.22}P_{h,0} = -h^2 \Delta -1= h^2 D_{x_{d}}^2  - R(x_{d}, x',hD_{x'})+ h Q(x, hD_{x}), 
\end{equation}
with $R$ a second order tangential operator and $Q$ a first order operator.\par
We recall now the usual decomposition of $T^*\partial \Omega$ (in this coordinate system). Denote by $r( x',x_d,\xi')$ the semi-classical principal symbol of $R$ and $r_{0}= r\mid _{x_{d}=0}$. Then $T^*\partial \Omega$ is the disjoint union of $\mathcal{ E}\cup \mathcal { G} \cup \mathcal{ H}$ with 
\begin{equation}
\label{3.24}\mathcal{ E}= \{ r_{0}<0\}, \mathcal{ G}=\{r_{0}=0\}, \mathcal{ H}= \{r_{0}>0\}.
\end{equation}
 Remark that $j$ gives a natural identification between $Z\mid_{\partial M}$ and ${\mathcal {H}}\cup {\mathcal {G}} \subset T^* \partial M$.
In $\mathcal{G}$ we distinguish between the {\em diffractive} points $\mathcal{ G}^{2,+}=\{ r_{0}=0, r_{1}= \partial _{x_{d}}r\mid _{x_{d}=0}>0\}$ and the {\em gliding} points $\mathcal{ G}^{-}=\{ r_{0}=0, r_{1}= \partial _{x_{d}}r\mid _{x_{d}=0}\leq 0\}$. We will make the assumption ($\Omega$ has no infinite order contact with its tangents) that for any $\varrho_{0}\in T^*\partial M$, there exists $N\in {\mathbb N}$ such that
$$H_{r_{0}}^N (r_{1})\neq 0$$\par
 The definition of the generalized bicharacteristic flow, $\varphi_{s}$ associated to the operator $P_0$ is essentially the definition given in~\cite{MeSj82}:
\begin{definition}
A generalized bicharacteristic curve $\gamma(s)$ is a continuous curve from an interval $I \subset {\mathbb R}$ to $ Z$ such that
\begin{enumerate}
\item if $s_{0}\in I$ and $\gamma(s_{0})\in T^* \Omega$ then close to $s_{0}$, $\gamma$ is an integral curve of the Hamiltonian vector field $H_{\widetilde p}$
\item If $s_{0}\in I$ and $\gamma(s_{0})\in {\mathcal {H}}\cup {\mathcal {G}}^{2,+}$ then there exists $\varepsilon>0$ such that for $0<|s-s_{0}|<\varepsilon$,  $x_{d}(\gamma(s))>0$
\item If $s_{0}\in I$ and $\gamma(s_{0})\in{\mathcal {G}}^{-}$ then for any function $f \in C^\infty(T^*{\mathbb R}^{d}\mid_{\overline{ \Omega}})$ satisfying the symmetry condition 
\begin{equation}
\label{eq.sym}
\forall \varrho_{0} \in Z, \forall \widehat {\varrho_{0}}, \widetilde{ \varrho_{0}}\in j^{-1}(\varrho_{0})\cap \text{Car} (\widetilde P), f(\widehat {\varrho_{0}}) = f(\widetilde{ \varrho_{0}})\end{equation}
then
$$\frac d {ds} f(j(\gamma(s))\mid_{s= s_{0}}= H_{\widetilde p}\mid_{j^{-1}( \gamma(s_{0}))} f(j^{-1}(\gamma( s_{0}))) $$
\end{enumerate}
\end{definition}
It is proved in~\cite{MeSj82} that under the assumption of no infinite order contact, through every point $\varrho_{o}\in {^b}T^* M \setminus \{0\}$ there exists a unique generalized bicharacteristic (which is furthermore a limit of bicharacteristics having only hyperbolic contacts with the boundary). This defines the flow $\Phi$. 
\subsection{Wigner measures}
Consider functions
$a=a_i+a_\partial $ with $a_{i}\in C^\infty_{0}(T^* M)$, and $a_{\partial}\in C^\infty_{0}({\mathbb R}^{2d-1})$. Such symbols are quantized in the following way:
take $\varphi_{i}\in C^\infty_{0}(M)$  (resp $\varphi_{\partial}\in C^\infty_{0}({\mathbb R}^d)$) equal to $1$ near the $x$-projection of $\text{supp}(a_{i})$ (resp the $x$-projection of $\text{supp}(a_{\partial})$) and define
\begin{multline}
\label{eq3.1bis}
\text{Op}_{h}^{\varphi_{i}, \varphi_{\partial}}(a)(x, hD_{x})f= \frac 1 { (2\pi h)^d} \int e^{ i(x-y)\cdot \xi/h}a_{i}(x, \xi)\varphi_{i}(y)f(y) dy d\xi \\
+ \frac 1 { (2\pi h)^{ d-1}} \int e^{ i(x'-y')\cdot \xi'/h}a_{\delta}(x_{d}, x', \xi)\varphi_{\delta}(x_{d}, y')f(x_{d}, y') dy' d\xi'.   
\end{multline}
Remark that according to the symbolic semi-classical calculus, the operator $\text{Op}_h^{\varphi_{i}, \varphi_{\partial}}(a)$ does not depend on the choice of functions $\varphi_{i}, \varphi_{\partial}$, modulo operators on $L^2$ of norms bounded by $O(h^\infty)$. For conciseness we shall in the sequel drop the index $\varphi_{i}, \varphi_{\partial}$.\par
Denote by ${\mathcal A}^h$ the space of the operators which are a finite sum of operators obtained as above in suitable coordinate systems near the boundary and for $B\in {\mathcal {A}}$, by $b=\sigma(B)$ the semiclassical symbol of the operator $A$. For such functions $b$ we can define  $\kappa(b)\in C^0(Z)$ by
\begin{equation} 
\kappa(b)(\rho )=b(j^{-1}(\rho )) \label{eq:9}
 \end{equation} 
(the value is independent of the choice of $j^{ -1}( \rho)$ since the operator is tangential). \par
The set 
\begin{equation} 
\{\kappa(b),b=\sigma (B), B\in {\mathcal A}^h\} \label{eq:10}
 \end{equation} is a locally dense subset of
$C^0_{c}(Z)$.\par
\subsubsection{Elliptic regularity}
The sequence $v_h$ satisfies (with $\chi_h = \chi (a/h)$)
\begin{align*}
P_hv_h=&\chi_hP_hu_h-h^2\mathrm{div}(\nabla\chi_h u_h)-h^2\nabla\chi_h\nabla u_h
-ih\mathrm{div}(a\nabla\chi_h u_h)-iha\nabla\chi_h\cdot\nabla u_h.
\end{align*}
Since $\nabla {\chi_h}=h^{-1}\chi'(a/h)\nabla a$ and $|\nabla a|\lesssim a^{\frac{1}{2}}$, by Corollary \ref{apriori}, we have
$$ P_hv_h=o_{L^2}(h)+o_{H^{-1}}(h^2). 
$$
Thus
$$ (h^2\Delta + 1) v_h=-ih\mathrm{div}(a\nabla(\chi_h u) )+o_{L^2}(h)+o_{H^{-1}}(h^2).$$
Using Corollary \ref{apriori} again and the fact that $a\lesssim h$ on the support of $\chi_h$, we deduce that 
$ha\nabla(\chi_hu)=o_{L^2}(h^{\frac{3}{2}})$, hence  
$$  (h^2\Delta + 1) v_h = o(h^{\frac{3}{2}})_{H^{-1}(\Omega)}+o_{L^2}(h).$$

We deduce, by standard elliptic regularity results 
\begin{prop}
If $a_{i}$ is equal to $0$ near $\mathrm{Car}(P_0)$ then 
\begin{equation}
\label{eq3.20}\lim_{k\rightarrow + \infty}\left(\Op_{h_k}(a_{i})v_{h_k}, v_{h_k}\right)_{ L^2}=0,
\end{equation}
\end{prop}

while near the boundary (see e.g.~\cite[Appendice A.1]{BuLe01} in a slightly different context) we get
\begin{prop}
If   $a_{\partial}$ is equal to $0$ near $Z$ (i.e. $ a_{i}$ is supported in the elliptic region)
 then
\begin{equation}
\label{eq3.21bis}\lim_{k\rightarrow + \infty}\left(a_\partial( x',x_d, h_{k}D_{x'})v_{h_k}, v_{h_k}\right)_{ L^2}=0.
\end{equation} 
\end{prop}

\begin{rema}
Note that if we regard the damping term $h\mathrm{div}(a\nabla(\chi_h u))=o_{H^{-1}}(h^{\frac{3}{2}})$ as a source term, we are not able to use the classical propagation theorem for $h^2\Delta+1$ as a black box, as such a strategy would require smaller r.h.s., namely $o_{H^{-1}}(h^2)+o_{L^2}(h)$. 
On the other hand, an integration by parts shows from Lemma~\ref{dampedregion} $$ \Bigl( \mathrm{div}(a\nabla(\chi_h u)), \chi_hu \Bigr) _{L^2} = -  \| a^{1/2} \nabla_x \chi_h u\| ^2_{L^2}= o(1),
$$ 
and this will ensure that in the propagation estimates such terms are invisible.  The key of our analysis in the sequel will be to systematically uses this procedure:  testing the damping term on expressions like $Q_h\chi_hu$,  doing the integration by part and then balancing $a^{\frac{1}{2}}$ to the other side. It is to perform this analysis that we need the condition $|\nabla a|\leq Ca^{\frac{1}{2}}$ to ensure the gain $O(h)$ from the commutator $[a^{\frac{1}{2}},Q_h]$. More precisely, we shall need the following lemma:
\end{rema}

\begin{lemma}\label{apriori2}
	Assume that $Q_h,B_{0,h},B_{1,h}$ are tangential $h$-pseudodifferential  operators of order $0$ and $B_h=B_{0,h}+B_{1,h}hD_{x_d}$, then
	$$ h^{-1}\big(Q_hM_hu,B_hu\big)_{L^2}=o(1),
	$$
	where $M_h=-h\mathrm{div} a\nabla $.
\end{lemma}
\begin{proof}
	Since $M_hu=-\nabla ah\nabla u-ah\Delta u$, from Corollary \ref{apriori}, we have
	$$ h^{-1}\big|\big(Q_h\nabla a h\nabla u, B_h u\big)_{L^2}\big|\leq C\|\nabla a\nabla u\|_{L^2}\|u\|_{H_h^1}=o(1).
	$$
	To estimate $\big(Q_ha\Delta u, B_hu\big)_{L^2}$, we write $Q_ha\Delta u=a^{\frac{1}{2}}Q_ha^{\frac{1}{2}}\Delta u+[Q_h,a^{\frac{1}{2}}]a^{\frac{1}{2}}\Delta u$. From Corollary \ref{apriori}, $a^{\frac{1}{2}}\Delta u=o_{L^2}(h^{-1})$. By Corollary \ref{Commutator:Lip}, $[Q_h,a^{\frac{1}{2}}], [B_h,a^{\frac{1}{2}}]=O_{\mathcal{L}(L^2)}(h)$. Therefore,
	$$ \big(Q_ha\Delta u, B_hu\big)_{L^2}=\big(Q_ha^{\frac{1}{2}}\Delta u, B_h a^{\frac{1}{2}}u\big)_{L^2}+o(1).
	$$
	Again from Corollary \ref{apriori}, we have $B_ha^{\frac{1}{2}}u=O_{L^2}(h)$, hence $(Q_ha\Delta u, B_hu)_{L^2}=o(1)$. The proof of Lemma \ref{apriori2} is complete.
\end{proof}  

\subsubsection{Definition of the measure}
The following results gives the existence of semi-classical measures.
\begin{prop}\label{prop:2.4}  Let $(v_{h_{k_p}})$ be a sequence bounded in $L^2( \Omega)$. There exists a subsequence $(k_{p})$ and a Radon positive measure $\mu $ on $Z$  such that 
\begin{equation} 
\forall  Q\in{\mathcal A}^{h_{k_p}}\quad \lim_{p\rightarrow \infty }(Qv_{h_{k_{p}}},v_{h_{k_{p}}})_{L^2}=\langle\mu ,\kappa(\sigma(Q))\rangle.
 \end{equation}
\end{prop}
The proof of this result relies on the G\aa rding inequality for tangential operators (see G. Lebeau~\cite{Le96} for a proof in the classical context and~\cite{GeLe93-1, Bu97-1} for the semi-classical construction). As before, we drop the indexes $k_p$ and denote by $(v_h)$ the extracted sequence.\par
\begin{prop}[First properties of the measure $\mu$]
\label{prop3.2}
We have
\begin{equation}
\label{eq3.27}\mu({\mathcal {H}})=0.
\end{equation}
Moreover, for any tangential symbol $b$, 
\begin{equation}
\label{eq3.27'}\limsup_{k\rightarrow + \infty}| \left(\Op_h(b) {h_{k}} D _{x_{d}} v_{h_k}, v_{h_k}\right)_{ L^2}|\leq C \sup_{\varrho\in \text{ supp}( b)}|r|^{ 1/2} |b| .
\end{equation}
\end{prop}

\begin{proof}
	The first property \eqref{eq3.27} follows from the fact that the trajectories near a hyperbolic point is transversal to the boundary. It follows from \cite{BuLe01}, with additional attention to the damping term $ih\mathrm{div} a\nabla v_h$. We factorize $P_{h,0}=-h^2\Delta-1$ as $(hD_{x_d}-L_h^{+})(hD_{x_d}-L_h^{-})+\mathcal{O}_{H^{\infty}}(h^{\infty})$ (see \cite[Lemme 6.1]{BuLe01}) near $\rho_0\in\mathcal{H}$ and choose $L_h^{\pm}$ with principal symbols $\pm l(x',{x_d},\xi')=\pm\sqrt{r(x',{x_d},\xi')}$.
	we denote by $q_0(x',\xi')\in C_c^{\infty}(\mathcal{H})$ and $q^{\pm}(y,x',\xi')$ solutions of 
	$$ \partial_{x_d}q^{\pm}\mp \{l,q^{\pm}\}=0,\quad q^{\pm}|_{{x_d}=0}=q_0.
	$$
	Denote by
	$$ u^{\pm}:=\psi({x_d})Q_h^{\pm}(hD_{x_d}-L_h^{\mp})u,
	$$
	where $\psi({x_d})\equiv 1$ if $0\leq x_d\leq \epsilon_0$.
	We have
	\begin{align*}
	(hD_{x_d}-L_h^{\pm})u^{\pm}=&\psi({x_d})[hD_{x_d}-L_h^{\pm},Q_h^{\pm}](hD_{x_d}-L_h^{\mp})u+Q_h^{\pm}f_h+\frac{h}{i}\psi'({x_d})Q_h^{\pm}(hD_{x_d}-L_{h}^{\mp})u\\+&iQ_h^{\pm}M_hu+O_{L^2}(h^{\infty})
	\end{align*} 
	where $M_h=h\mathrm{div} a\nabla $ and $f_h=o_{L^2}(h)$. By definition of $q^{\pm}$, the first term of r.h.s. is $O(h^2)$, hence 
	\begin{align}\label{eq:hyperbolic}
	(hD_{x_d}-L_h^{\pm})u^{\pm}=g_h^{\pm}-ih\psi'({x_d})Q_h^{\pm}(hD_{x_d}-L_h^{\mp})u+iQ_h^{\pm}M_hu,\quad g_h^{\pm}=o_{L^2}(h).
	\end{align}
	We have
	\begin{align*}
	&h\frac{d}{d{x_d}}(u^{\pm},u^{\pm})_{L^2(\partial)}= -2\Im\big(g_{h}^{\pm}-ih\psi'({x_d})Q_h^{\pm}(hD_{x_d}-L_h^{\mp})u +iQ_h^{\pm}M_hu,u^{\pm} \big)_{L^2(\partial)}\!\!+\!\!i\big((L_{h}^{\pm}-L_h^{\pm,*})u^{\pm},u^{\pm} \big)_{L^2(\partial)}.
	\end{align*}
	For $y_0\leq \epsilon_0$, we have
	\begin{align*}
	\|u^{\pm}(y_0)\|_{L^2(\partial)}^2\leq&\|u^{\pm}(0)\|_{L^2({x_d}=0)}^2+Ch^{-1}\|g_h^{\pm}\|_{L^2({x_d}\leq y_0)}\|u^{\pm}\|_{L^2({x_d}\leq y_0)} +C\|u^{\pm}\|_{L^2({x_d}\leq y_0)}^2\\
	+&Ch^{-1}\big|(Q_h^{\pm}M_hu,u^{\pm})_{L^2({x_d}\leq y_0)}\big|
	\end{align*} 
	The second line of r.h.s. is $o(1)$, due to Lemma \ref{apriori2}, and the first line of r.h.s. can be bounded by 
	$$\|u^{\pm}(0)\|_{L^2({x_d}=0)}^2+C\|u^{\pm}\|_{L^2({x_d}\leq y_0)}^2+o(1).$$
	 Integrating both sides over $y_0\leq \epsilon_0$, letting $h\rightarrow 0$ and then $\epsilon_0\rightarrow 0$, we deduce that $\langle\mu\mathbf{1}_{y=0},q_0\rangle=0$. This proves \eqref{eq3.27}. 
	
	For \eqref{eq3.27'}, it suffices to prove the inequality for $u$ instead of $v$. By Cauchy-Schwarz, 
	$$ \big|\big(\Op_h(b)h\partial_{x_d}u,u \big)_{L^2}\big|\leq \big|\big(\Op_h(b)^*\Op_h(b)h\partial_{x_d}u,h\partial_{x_d}u \big)_{L^2}\big|^{\frac{1}{2}}\|u\|_{L^2}.
	$$
	Doing integration by part, we have
	\begin{align*}
	\big(\Op_h(b)^*\Op_h(b)h\partial_{x_d}u,h\partial_{x_d}u\big)_{L^2}=-\big(\Op_h(b)^*\Op_h(b)h^2\partial_{x_d}^2u,u\big)_{L^2}+O(h).
	\end{align*}
	Replacing $h^2\partial_{x_d}^2u$ by equation
	$ h^2\partial_{x_d}^2u=-R_hu-iM_hu+O_{L^2}(h), 
	$ we deduce that
	$$ \big|\big(\Op_h(b)^*\Op_h(b)h^2\partial_{x_d}^2u,u\big)_{L^2}\big|\leq \big|\big(\Op_h(|b|^2)R_hu,u\big)_{L^2}\big|+O(h)+\big|\big(\Op_h(b)^*\Op_h(b)M_hu,u\big)_{L^2}\big|.
	$$
	From Lemma \ref{apriori2}, the third term of r.h.s. is $o(h)$. Passing $h\rightarrow 0$, we complete the proof of Proposition \ref{prop3.2}.
\end{proof}

\subsection{Invariance of the measure}
The key to prove the invariance of the measure will be to apply  the propagation theorem in ~\cite[Th\'eor\`eme 1]{BuLe01}. 
\begin{theo*}\label{BL} The two following properties are equivalent
\begin{enumerate}
\item The measure $\mu$ is invariant along the generalised flow.
\item The measure $\mu$ satisfies $\dot{\mu} =0$ and $\mu ( \mathcal{G}_2^+) =0$ in the sense that $\langle\mu,\{p,q \}\rangle=0$ holds for any even symbol $q\in C_c^{\infty}(\mathrm{Car}(P_0))$, i.e. $q(x',x_d=0,\xi',\xi_d)=q(x',x_d=0,\xi',-\xi_d)$.
\end{enumerate}
\end{theo*}
\begin{rema}
Technically, Theorem~\ref{BL} is proved in~\cite{BuLe01}
for time dependent measures, i.e. measures depending in addition on two additional variables $(t, \tau)\in T^* \mathbb{R}$, and $p$ is replaced by $p- \tau^2$. However, it is easy to apply the results from~\cite{BuLe01} by considering the measure
\begin{equation}\label{particular} \widetilde{\mu} = \mu_{x, \xi} \otimes dt \otimes \delta_{\tau =1},
\end{equation}
which is supported in 
$\mathrm{Car}(-\Delta+ \partial_t^2)$
and satisfies $\dot{\widetilde{\mu}} =0$ in the sense that $\langle\mu,\{p- \tau^2,q \}\rangle=0$ holds for any even symbol $q\in C_c^{\infty}(\mathrm{Car}(P_0- \tau^2))$, i.e. $q(x',x_d=0,t, \xi',\xi_d, \tau )=q(x',x_d=0, t,\xi',-\xi_d, \tau)$. Remark that though we shall not use it, the measure $\widetilde{\mu}$ is, in the sense of~\cite[Section 2]{BuLe01}, the microlocal defect measure on the sequence $v_n(t,x) = h_n e^{ith_n^{-1}}u_n (x)$ (the pre-factor $h_n$ comes from the $H^1$ normalisation of the sequence $v_n$  in~\cite{BuLe01}).
Now, the generalised bicharacteristic flow for $p_0 - \tau^2$, $\Psi_s$ is given in terms of the generalised bicharacteristic flow for $p_0 $, $\psi_s$ by 
$$ \Psi_s (t, x, \tau=1 , \xi)= (t-2s, \tau =1, \psi_s(x, \xi)),$$
The set of diffractive points $\widetilde{\mathcal{G}}^{2,+} $ in the time dependent frame-work is given by 
$$ \widetilde{\mathcal{G}}^{2,+} = {\mathcal{G}}^{2,+}\times \mathbb{R} \times \{\tau =\pm 1\}
$$
and consequently, 
$$ \mu( {\mathcal{G}}^{2,+} )=0 \Leftrightarrow \widetilde{\mu} (\widetilde{\mathcal{G}}^{2,+} ) =0,$$
and in view of the particular form~\eqref{particular}, the invariance of $\widetilde{\mu}$ by $\Psi_s$ is equivalent to the invariance of $\mu$ by $\psi_s$.
\end{rema} 
Let us now  briefly explain the procedure we are going to follow.
\begin{itemize}
	\item First from Proposition \ref{prop:2.4} and the elllipticity (Proposition \ref{eq3.20}, Proposition \ref{eq3.21bis}), the measure $\mu$ is defined on $Z=j(\mathrm{Car}(P_0))$ by testing on symbols of the form $q=q_{i}+q_{\partial}, q_i\in C_c^{\infty}(T^*\Omega)$ and $q_{\partial}$ tangential (which is dense in $C_0(Z)$). 
	\item Using the fact $\mu(\mathcal{H})=0$ (Proposition \ref{prop3.2}), the measure $\mu$ can be extended to test on functions of $\mathrm{Car}(P_0)$ which admits a representation (thanks to Malgrange's theorem)
	$$ q(x',x_d,\xi',\xi_d)=q_0(x',x_d,\xi')+\xi_dq_1(x',x_d,\xi'),\; \text{ on } \xi_d^2=r(x',x_d,\xi').
	$$
	Then, we will show in Proposition \ref{extension} that for tangential $h$-pseudodifferential operators $B_{0,h}, B_{1,h}$, the quadratic form $$((B_{0,h_k}+B_{1,h_k}\frac{1}{i}h_k\partial_{x_d}))v_{h_k},v_{h_k})$$ converges to $\langle\mu,b_0+b_1\xi_d\mathbf{1}_{\rho\notin\mathcal{H}}\rangle$, by a suitable limit procedure for symbols in $\mathcal{A}^h$. Consequently, for any $q\in C_c^{\infty}(\mathrm{Car}(P_0))$, we can make sense of the expression
	$$ \langle\mu,\{p,q\}\rangle
	$$
	$\mu$-a.e., by viewing $\{p,q\}=2\xi_d\partial_{x_d}q\mathbf{1}_{\rho\notin\mathcal{H}}-\{r,q\}.$ We remark that to calculate $\{p,q\}$, it is enough to choose one representation $q=q_0+q_1\xi_d$ on $\mathrm{Car}(P_0)$, since $\{p,p\}=0$ and $p=0$ on supp$(\mu)$.
	\item Finally, to prove that the measure $\mu$ is invariant along the Melrose-Sj\"ostrand flow, we apply Theorem~\ref{BL}, for which we need to verify the following conditions:
	\begin{itemize}
		\item[(a)] $\mu(\mathcal{G}^{2,+})=0$
		\item[(b)] $\dot{\mu}=0$, in the sense that $\langle\mu,\{p,q \}\rangle=0$ holds for any even symbol $q\in C_c^{\infty}(\mathrm{Car}(P_0))$, i.e. $q(x',x_d=0,\xi',\xi_d)=q(x',x_d=0,\xi',-\xi_d)$. 
	\end{itemize}
\end{itemize}

The verification of (a),(b) in our context is based on the propagation formula: Proposition \ref{propagation:interior} and Proposition~\ref{propagation:boundary}. Especially, starting from Proposition \ref{propagation:boundary}, by choosing suitable test symbols of the form $q_0+q_1\xi_d$, we are able to verify the conditions (a) and (b).




\begin{prop}\label{extension} 
If $B_{0,h},B_{1,h}$ are two tangential $h$-pseudodifferential operators of with principal symbols $b_0,b_1$ of order $0$, then we have
\begin{align*}
\lim_{k\rightarrow\infty}((B_{0,h_k}+B_{1,h_k}\frac{1}{i}h_k\partial_{x_d})v_{h_k},v_{h_k} )_{L^2}=\langle\mu,b_0+b_1\xi_d\mathbf{1}_{\rho\notin\mathcal{H}} \rangle.
\end{align*} 
\end{prop}
\begin{proof}
Since $B_{0,h}$ and $B_{1,h}$ are all tangential, by the definition of the measure, the first term $(B_{0,h_k}v_{h_k},v_{h_k})_{L^2}$ converges to $\langle\mu,b_0\rangle$. It remains to prove the convergence of the second term $(B_{1,h_k}\frac{1}{i}h_k\partial_{x_d}v_{h_k},v_{h_k})_{L^2}$. For this, we pick $\epsilon>0, \delta>0$ and define
\begin{align*} &B_{1,h_k,\epsilon}=\big(1-\psi\big(\frac{x_d}{\epsilon}\big)\big)B_{1,h_k}\big(1-\psi\big(\frac{x_d}{2\epsilon}\big)\big),\quad B_{1,h_k}^{\epsilon}=B_{1,h_k}-B_{1,h_k,\epsilon},\\
& B_{1,h_k}^{\epsilon,\delta}=\Op_{h_k}\big(\psi\big(\frac{r}{\delta}\big)\big)B_{1,h_k}^{\epsilon},\quad B_{1,h_k,\delta}^{\epsilon}=B_{1,h_k}^{\epsilon}-B_{1,h_k}^{\epsilon,\delta},
\end{align*}
where $\psi$ is a cutoff function which is $1$ near $0$. Now by the definition of $\mu$ and the dominating convergence,
$$ \lim_{\epsilon\rightarrow 0}\lim_{k\rightarrow \infty}(B_{1,h_k,\epsilon}\frac{1}{i}h_k\partial_{x_d}v_{h_k},v_{h_k})_{L^2}=\langle\mu, b_1\xi_d\mathbf{1}_{x_d>0}\rangle=\langle\mu,b_1\xi_d\mathbf{1}_{\rho\notin\mathcal{H}}\rangle,
$$
since $\mu(\mathcal{E})=\mu(\mathcal{H})=0$. Now from Proposition \ref{prop3.2}, the contribution of 
$$ \lim_{\epsilon\rightarrow 0}\limsup_{k\rightarrow\infty}|(B_{1,h_k}^{\epsilon,\delta}h_k\partial_{x_d}v_{h_k},v_{h_k}  )_{L^2}|\leq C\delta^{\frac{1}{2}},
$$
which converges to $0$ if we let $\delta\rightarrow 0$. Finally, by Cauchy-Schwarz, 
$$ |(B_{1,h_k,\delta}^{\epsilon}h_k\partial_{x_d}v_{h_k},v_{h_k} )_{L^2}|\leq \|h_k\partial_{x_d}v_{h_k}\|_{L^2}\|
(B_{1,h_k}^{\epsilon,\delta})^*v_{h_k}
\|_{L^2}.
$$
Notice that 
\begin{align*}
&\lim_{k\rightarrow\infty}\|
(B_{1,h_k}^{\epsilon,\delta})^*v_{h_k}
\|_{L^2}^2=\lim_{k\rightarrow\infty}(B_{1,h_k}^{\epsilon,\delta}(B_{1,h_k}^{\epsilon,\delta})^*v_{h_k},v_{h_k})_{L^2}\\
=&\langle\mu,\big(1-\psi\big(\frac{r}{\delta}\big)\big)\big[\psi\big(\frac{x_d}{\epsilon}\big)\big(1-\psi\big(\frac{x_d}{2\epsilon}\big)\big)b_1+\big(1-\psi\big(\frac{x_d}{\epsilon}\big)\psi\big(\frac{x_d}{2\epsilon}\big) \big)b_1 \big] \rangle,
\end{align*}
taking the double limit $\limsup_{\delta\rightarrow 0}\limsup_{\epsilon\rightarrow 0}$, we obtain that
$$\limsup_{\delta\rightarrow 0}\limsup_{\epsilon\rightarrow 0}\lim_{k\rightarrow\infty}\|(B_{1,h_k}^{\epsilon,\delta})^*v_{h_k}\|_{L^2}^2\leq \langle\mu, b_1^2\mathbf{1}_{x_d=0}\mathbf{1}_{r\neq 0}\rangle=0,
$$
since $\mu\mathbf{1}_{\mathcal{E}\cup\mathcal{H}}=0$. This completes the proof of Proposition \ref{extension}.
\end{proof}


\section{Interior propagation estimate}\label{sec.5}

\begin{prop}[Interior propagation]\label{propagation:interior}
	Let $Q_h=\widetilde{\chi}Q_h\widetilde{\chi}$ be a $h$-pseudodifferential operator of order 0, where $\widetilde{\chi}\in C_c^{\infty}(\Omega)$, then we have
	$$ \frac{1}{ih}\big([h^2\Delta+1,Q_h]v_h,v_h \big)_{L^2}=o(1).
	$$	
\end{prop}

\begin{proof}
Denote by $P_h=P_{h,0}+iM_{h}$ with $M_{h}=-h\div a\nabla$ and $P_{h,0}=-h^2\Delta-1$, we have
\begin{align*}
\frac{1}{ih}\big([P_{h,0},Q_h]v,v \big)_{L^2}=&\frac{1}{ih}\big(Q_hv,P_{h,0}v\big)_{L^2}-\frac{1}{ih}\big(P_{h,0}v,Q_h^*v\big)_{L^2}\\
=&\frac{1}{ih}\big(Q_hv,\chi P_{h,0}u\big)_{L^2}-\frac{1}{ih}\big(\chi P_{h,0}u,Q_h^*v\big)_{L^2} + \mathcal{R}_1 \\
\text{with }\mathcal{R}_1= \frac{1}{ih}&\big(Q_hv,[P_{h,0},\chi]u\big)_{L^2}-\frac{1}{ih}\big([P_{h,0},\chi]u,Q_h^*v \big)_{L^2}.
\end{align*}
By using the equation $P_{h,0}u=f-iM_{h}u$, we have
\begin{align}\label{eq:1-Prop2.7}
\frac{1}{ih}\big([P_{h,0},Q_h]v,v\big)_{L^2}=&\mathcal{R}_1+\mathcal{R}_2+o(1),
\end{align}
where
$$ \mathcal{R}_2=\frac{1}{h}\big(Q_hv,\chi M_{h}u \big)_{L^2}+\frac{1}{h}(\chi M_{h}u,Q_h^*v  )_{L^2}.
$$
Note that
$$ -[P_{h,0},\chi]=h\nabla\big(\nabla a \chi'(a/h ) \big)+2\nabla a \chi'(a/h)h\nabla,
$$
since $h\nabla(\chi(a/h) )=\nabla a\chi'(a/h )$.\\
\noindent
$\bullet$ {\bf Claim 1: $\mathcal{R}_1=o(1)$}\\
It suffices to show that $ih^{-1}\big(B_hv,[P_{h,0},\chi]u\big)_{L^2}=o(1)$ for any compact supported $h$-pseudo $B_h$ of degree 0. By integration by part, 
\begin{align*}
-\frac{1}{ih}\big(B_hv,[P_{h,0},\chi]u\big)_{L^2}=&-\frac{1}{ih}\big( B_h v , h \div ( \chi' ( \frac a h) \nabla a u )+ \chi' (\frac a h ) \nabla a h \nabla u  \big)_{L^2}
\end{align*}
and  we simply apply Corollary \ref{apriori}, to get for each term  $o(1)$. \noindent
\\
$\bullet$ {\bf Claim 2: $\mathcal{R}_2=o(1)$}\\
It suffices to prove that $(Q_hv,\chi\div(a\nabla u))_{L^2}=o(1)$. We write
\begin{align*}
\big(Q_hv,\chi\div a\nabla u\big)_{L^2}=&-\big((\nabla\chi)Q_hv,a\nabla u\big)_{L^2}-\big(\chi [\nabla,Q_h]v,a\nabla u\big)_{L^2}
-\big(\chi Q_h \nabla v,a\nabla u \big)_{L^2}.
\end{align*}
Since $|a^{\frac{1}{2}}\nabla\chi|=h^{-1}|a^{\frac{1}{2}}\chi'\nabla a|\leq C$, from Corollary \ref{apriori}, the first term of r.h.s. can be bounded by 
$$\|Q_hv\|_{L^2}\|a^{\frac{1}{2}}\nabla u\|_{L^2}=o(1).
$$
 The second term of r.h.s. can be bounded by $o(h)$. Observe that $\nabla(a^{\frac{1}{2}})=\frac{1}{2}a^{-\frac{1}{2}}\nabla a$ is bounded, thus from Corollary \ref{Commutator:Lip},
$$ [a^{\frac{1}{2}},Q_h]=O_{\mathcal{L}(L^2)}(h).
$$
Therefore,
$$\big|\big(\chi Q_h \nabla v,a\nabla u \big)_{L^2}\big| \leq \big|\big(\chi Q_h a^{\frac{1}{2}}\nabla v, a^{\frac{1}{2}}\nabla u \big)_{L^2}\big|+\big|\big(\chi [ a^{\frac{1}{2}},Q_h]\nabla v, a^{\frac{1}{2}}\nabla u \big)_{L^2}\big|.
$$
The second term is bounded by $Ch\|\nabla v\|_{L^2}\|a^{\frac{1}{2}}\nabla u\|_{L^2}=o(1)$, and the first term can be bounded by $o(1)$, due to Lemma \ref{dampedregion}. This completes the proof of Proposition \ref{propagation:interior}.
\end{proof}	

\section{Propagation near the boundary}\label{sec.7}

Recall that $v_h=\chi(a/h)u_h$. Consider the operator
$$ B_h=B_{0,h}+B_{1,h}\frac{h}{i}\partial_{x_d}
$$
where $B_{j,h}=\widetilde{\chi}_1\Op_h(b_j)\widetilde{\chi}_1$, $j=0,1$ are two tangential operators and  $\widetilde{\chi}_1$ has compact support near a point $z_0\in\partial\Omega$. Note that in the local coordinate system,
$$ P_{h,0}=-h^2\Delta-1=-\frac{1}{\sqrt{|g|}}h\partial_{x_d}\sqrt{|g|}h\partial_{x_d}-R_h,
$$
where $R_h$ is a self-adjoint tangential operator of order $2$. The operator involving the damping can be written as
$$ M_h=-\frac{h}{\sqrt{|g|}}\partial_{x_d}\sqrt{|g|}a\partial_{x_d} -\frac{h}{\sqrt{|g|}}\partial_{x'_k}\sqrt{|g|}ag^{jk}\partial_{x'_j}
$$

\begin{prop}[Boundary propagation]\label{propagation:boundary} 
$$ \frac{1}{ih}\big([P_{h,0},B_h]v,v\big)_{L^2}=\big(B_{1,h}|_{x_d=0}(h\partial_{x_d}v)|_{x_d=0},(h\partial_{x_d}v)|_{x_d=0} \big)_{L^2(\partial)}+o(1).
$$	
\end{prop}

\begin{proof}
We give the proof in the case $B_{0,h} = 0$. The $B_{0,h}$ terms are handled
by slightly simpler versions of the same computations. By developing the commutator, we have
\begin{align*}
\frac{1}{ih}\big([P_{h,0},B_h]v,v\big)_{L^2}=&\frac{1}{ih}\big(B_hv,P_{h,0}v \big)_{L^2}-\frac{1}{ih}\big(B_hP_{h,0}v,v\big)_{L^2}+\big(B_{1,h}|_{x_d=0}(h\partial_{x_d}v)|_{y=0},h\partial_{x_d}v|_{x_d=0}\big)_{L^2(\partial)},
\end{align*}
where the boundary term (the third) comes from the integration by part of the term $$\frac{1}{ih}\big(\frac{1}{\sqrt{|g|}}h\partial_{x_d}\sqrt{|g|}h\partial_{x_d} v,v \big)_{L^2},$$ since $R_{h}$ is self-adjoint tangential operator. It suffices to show that
\begin{align}\label{o1}
\mathrm{I}_h:=\frac{1}{ih}\big(B_{1,h}h\partial_{x_d}v,P_{h,0}v \big)_{L^2}-\frac{1}{ih}\big(B_{1,h}h\partial_{x_d}P_{h,0}v,v\big)_{L^2}=o(1).
\end{align}
Since $v=\chi u$ and $P_{h,0}u=P_hu-iM_hu=f_h-iM_h u$, we have
$$ P_{h,0}v=\chi P_{h,0}u+[P_{h,0},\chi]u=\chi f_h-i\chi M_hu+[P_{h,0},\chi]u.
$$
Therefore,
$$ \mathrm{I}_h=o(1)+\mathrm{I}_{h,1}+\mathrm{I}_{h,2},
$$
where
$$\mathrm{I}_{h,1}=\frac{1}{ih}\big(B_{1,h}h\partial_{x_d} v, [P_{h,0},\chi]u\big)_{L^2}-\frac{1}{ih}\big(B_{1,h}h\partial_{x_d}[P_{h,0},\chi]u,v \big)_{L^2}
$$ 
and
$$ \mathrm{I}_{h,2}=\frac{1}{h}\big(B_{1,h}h\partial_{x_d} v, \chi M_hu \big)_{L^2}-\frac{1}{h}\big(B_{1,h}h\partial_{x_d} \chi M_{h}u,v  \big)_{L^2}
$$
\\
\noindent
$\bullet$ {\bf Claim 1: $I_{h,1}=o(1)$.}\\
Indeed, from integration by part, the second term $$ih^{-1}(B_{1,h}h\partial_{x_d}[P_{h,0},\chi]u,v)_{L^2}=ih^{-1}([P_{h,0},\chi]u,h\partial_{x_d}A_hv)_{L^2}$$ for some tangential operator $A_h$, hence it has the same structure as the first term. It suffices to show that
$$ h^{-1}\big(B_{1,h}h\partial_{x_d}v,[P_{h,0},\chi]u \big)_{L^2}=o(1).
$$ 
Since
$$ -[P_{h,0},\chi]u=h\nabla\cdot(\nabla a\chi'(a/h) )u+2\nabla a\chi'(a/h)\cdot h\nabla u,
$$
doing integration by part, we obtain that
\begin{align*}
 -h^{-1}\big(B_{1,h}h\partial_{x_d}v,[P_{h,0},\chi]u\big)_{L^2}=&-\big(\nabla(\overline{u}B_{1,h}h\partial_{x_d} v),\nabla a\chi' \big)_{L^2}+2h^{-1}\big(B_{1,h}h\partial_{x_d} v, \nabla a\chi' h\nabla u\big)_{L^2}\\
 =&-\big(\nabla B_{1,h}h\partial_{x_d} v,\nabla a\chi'  u\big)_{L^2}+h^{-1}\big(B_{1,h}h\partial_{x_d} v, \nabla a\chi' h\nabla u\big)_{L^2}.
\end{align*}
Note that $v=\chi u$, if one of the derivatives $h\partial_{x_d}$, $h\nabla $ fall on $\chi(a/h)$ we can bound them from Corollary \ref{apriori} by $o(h)$. If all the derivatives fall on $u$ in anyone of the two terms, by Lemma \ref{eenergy} and Corollary \ref{apriori}, these terms can be bounded by
		$$ \|h\nabla\partial_{x_d} u\|_{L^2}\|\nabla a u\|_{L^2}+h^{-1}\|h\partial_{x_d} u\|_{L^2}\|\nabla ah\nabla u\|_{L^2}=o(1).
		$$


\noindent
$\bullet$ {\bf Claim 2: $\mathrm{I}_{h,2}=o(1)$. }\\
It suffices to prove that
$$ h^{-1}\big(B_{1,h}h\partial_{x_d}(\chi u), \chi M_hu \big)_{L^2}=o(1). 
$$
Note that $-M_hu=\nabla ah \nabla u+ah\Delta u=o_{L^2}(1)$ and $h\partial_{x_d}(\chi u)=\partial_{x_d}a\chi'u+\chi h\partial_{x_d}u$. We have
$$ h^{-1}|\big(B_{1,h} \partial_{x_d} a\chi'u,\chi M_hu  \big)_{L^2}|\leq h^{-1} \|\nabla a u\|_{L^2}\|\chi M_hu\|_{L^2}=o(1),
$$
since $\|\nabla a  u\|_{L^2}=o(h)$ from Corollary \ref{apriori}. It remains to show that
$$ h^{-1}\big(B_{1,h}\chi h\partial_{x_d}u,\chi(\nabla a\cdot h\nabla u+ah\Delta u ) \big)_{L^2}=o(1).
$$
Since $\|\nabla a h\nabla u\|_{L^2}=o(h)$, we have $h^{-1}\big(B_{1,h}\chi h\partial_{x_d}u,\chi(\nabla a\cdot h\nabla u) \big)_{L^2}=o(1)$. Finally, we show that
$$ h^{-1}\big(B_{1,h}\chi h\partial_{x_d}u, \chi ah\Delta u \big)_{L^2}=o(1).
$$  Recall that from $|\nabla(a^{\frac{1}{2}})|\leq C$ and Corollary \ref{Commutator:Lip},
$$ [B_{1,h},a^{\frac{1}{2}}]=O_{\mathcal{L}(L^2)}(h),
$$
we have
\begin{align*}
h^{-1}|\big(B_{1,h}\chi h\partial_{x_d}u, \chi ah\Delta u \big)_{L^2}| \leq &h^{-1}|\big(B_{1,h}a^{\frac{1}{2}}\chi h\partial_{x_d}u, \chi a^{\frac{1}{2}}h\Delta u \big)_{L^2}|+h^{-1}|\big([B_{1,h},a^{\frac{1}{2}}]\chi h\partial_{x_d}u, \chi a^{\frac{1}{2}}h\Delta u \big)_{L^2}|\\
\leq & Ch^{-1}\|a^{\frac{1}{2}}h\nabla u\|_{L^2}\|a^{\frac{1}{2}}h\Delta u\|_{L^2}+C\|h\nabla u\|_{L^2}\|a^{\frac{1}{2}}h\Delta u\|_{L^2}=o(1).
\end{align*}
This completes the proof of Proposition \ref{propagation:boundary}.
\end{proof}


To show that the semi-classical measure $\mu$ of $(v_{h_k})$ is invariant along the Melrose-Sj\"ostrand flow (to complete the proof of Proposition \ref{resolvent-HFsemi}), we need to
verify the condition (2) in Theorem \ref{BL}. We will make use of the propagation formula, i.e. Proposition \ref{propagation:boundary}. Formally, for $B_h=B_{0,h}+B_{1,h}\frac{1}{i}h\partial_{x_d}$, the principal symbol of $\frac{i}{h}[P_{h,0},B_h]$ is given by $$\{\eta^2-r,b_0+b_1\xi_d \}=a_0+a_1\xi_d+a_2\xi_d^2,
$$
where
\begin{align}\label{relation}
a_0=b_1\partial_{x_d}r-\{r,b_0\}',\quad a_1=2\partial_{x_d}b_0-\{r,b_1\}',\quad a_2=2\partial_{x_d}b_1,
\end{align}
and $\{\cdot,\cdot\}'$ is the Poisson bracket for $(x',\xi')$ variables. On the other hand, by calculating the commutator, we find
\begin{align}\label{Malagrange}
\frac{i}{h}[P_{h,0},B_h]=A_0+A_1hD_{x_d}+A_2h^2D_{x_d}^2+h\Op_h(S_{\partial}^{0}+S_{\partial}^0\xi_d),
\end{align} 
where $A_0,A_1,A_2$ are tangential operators with symbols $a_0,a_1,a_2$, with respectively. We will prove the following propagation formula:
\begin{coro}\label{propagationformula}
	Assume that $B_h=B_{h,0}+B_{h,1}hD_{x_d}$, where $B_{h,0}, B_{h,1}$ are tangential operators of order 0 with symbols $b_0, b_1$, with respectively. Assume that $b=b_0+b_1\xi_d$. Define the formal Poisson bracket
	$$ \{p,b\}=(a_0+a_2r)+a_1\xi_d\mathbf{1}_{\rho\notin\mathcal{H}},
	$$	
	where $a_0,a_1,a_2$ are given by \eqref{relation}.
	Then the defect measure $\mu$ satisfies the equation
	$$ \langle\mu,\{p,b\}\rangle=-\langle\nu_{\partial},b_1 \rangle,
	$$
	where $\nu_{\partial}$ is the semiclassical measure of  $(h\partial_{x_d}v_{h}|_{x_d=0}).$
	Moreover, if $b$ is an even symbol (i.e. $b(x',x_d=0,\xi',\xi_d)=b(x',x_d=0,\xi',-\xi_d)$), then we have
	$$ \langle\mu,\{p,b\} \rangle=0.
	$$
	In particular, by combining Proposition~\ref{propagation:interior}, we have $\dot{\mu}=0$.
\end{coro}
\begin{proof}
	From Proposition \ref{propagation:boundary} and the decomposition \eqref{Malagrange}, we have
	\begin{align}
	\big( (A_0+A_1h_kD_{x_d}+A_2h_k^2D_{x_d}^2)v_{h_k},v_{h_k} \big)_{L^2}=-\langle\nu_{\partial},b_1 \rangle+o(1).
	\end{align}
	From Lemma \ref{dampedregion}, we can also replace the function $v_{h_k}$ on the l.h.s. by $u_{h_k}$. Using the equation of $u_{h_k}$:
	$$ (h^2D_{x_d}^2-R_{h_k})u_{h_k}=iM_{h_k}u_{h_k}-f_{h_k}+O_{L^2}(h_k),
	$$ 
	we deduce that
	$$ (A_2h_k^2D_{x_d}^2u_{h_k},u_{h_k})=(A_2R_{h_k}u_{h_k},u_{h_k})_{L^2}+o(1),
	$$
	thanks to Lemma \ref{apriori2}. Therefore, from Proposition \ref{extension},
	$$ \lim_{kh\rightarrow \infty}((A_0+A_1h_kD_{x_d}+A_2R_{h_k})u_{h_k},u_{h_k} )_{L^2}=\langle\mu,a_0+a_1\xi_d\mathbf{1}_{\rho\notin\mathcal{H}}+a_2r\rangle=\langle\mu,\{p,b\}\rangle.
	$$
Now if $b=b_0+b_1\xi_d$ is an even symbol, we must have $b_1|_{x_d=0}=0$, therefore, $\langle\mu,\{p,b\}\rangle=-\langle\nu_{\partial},b_1\rangle=0$. The proof of Lemma \ref{propagationformula} is complete.
\end{proof}
\begin{coro}\label{nochargingG2+}
	We have $\mu(\mathcal{G}^{2,+})=0$.
\end{coro}
\begin{proof}
	We will make use of the formula
	$ \langle\mu,\{p,b\}\rangle=-\langle\nu_{\partial},b_1\rangle
	$	
	by choosing $b=b_{1,\epsilon}\eta$ with
	$$ b_{1,\epsilon}(x',x_d,\xi')=\psi\big(\frac{x_d}{\epsilon^{\frac{1}{2}}}\big)\psi\big(\frac{r(x_d,x',\xi')}{\epsilon}\big)\kappa(x_d,x',\xi'),
	$$ 
	where $\psi\in C_c^{\infty}(\mathbb{R})$ equals to $1$ near the origin and $\kappa(y,x',\xi')\geq 0$ near a point $\rho_0\in\mathcal{G}^{2,+}$. Since $\{p,b_{\epsilon} \}=(a_{0}+a_2r)+a_1\xi_d\mathbf{1}_{\rho\notin\mathcal{H}}$, and $a_0,a_1,a_2$ are given by the relation \eqref{relation}. In particular for our choice, by direct calculation we have
	$$ a_0=b_{1,\epsilon}\partial_{x_d} r, \quad a_1=-\{r,\kappa\}'\psi\big(\frac{x_d}{\epsilon^{\frac{1}{2}}}\big)\psi\big(\frac{r}{\epsilon}\big),$$
	and
	$$ a_2=2\partial_{x_d}b_{1,\epsilon}=2\epsilon^{-\frac{1}{2}}\psi'\big(\frac{x_d}{\epsilon^{\frac{1}{2}}}\big)\psi\big(\frac{r}{\epsilon}\big)\kappa+2\frac{\partial_{x_d}r}{\epsilon}\psi\big(\frac{x_d}{\epsilon^{\frac{1}{2}}}\big)\psi'\big(\frac{r}{\epsilon}\big)\kappa+2\psi\big(\frac{x_d}{\epsilon^{\frac{1}{2}}}\big)\psi\big(\frac{r}{\epsilon}\big)\partial_{x_d}\kappa.
	$$
	Note that $a_2$ is uniformly bounded in $\epsilon$ and for any fixed $(y,x',\xi')$, $ra_2\rightarrow 0$ as $\epsilon\rightarrow 0$. Thus by dominating convergence, we have
	$$ \lim_{\epsilon\rightarrow 0}\langle\mu,\{p,b_{\epsilon} \}\rangle=\langle\mu,\kappa|_{x_d=0}\partial_{x_d}r\mathbf{1}_{r=0}\rangle\geq 0
	$$
	since $\partial_{x_d}r>0$ on $\mathcal{G}^{2,+}$. However,  $-\langle\nu_{\partial},b_{\epsilon}\rangle\leq 0$, we must have $\mu\mathbf{1}_{\mathcal{G}^{2,+}}=0$. This completes the proof of Lemma \ref{nochargingG2+}.
\end{proof}
From Lemma \ref{propagationformula} and Lemma \ref{nochargingG2+}, we have verified that $\dot{\mu}=0$ and $\mu(\mathcal{G}^{2,+})=0$, thus from Theorem \ref{BL}, the semi-classical $\mu$ is invariant along the Melrose-Sj\"ostrand flow. Thanks to the geometric control condition and the fact that $a^{\frac{1}{2}}v_{h_k}=o_{L^2}(1)$, we deduce that $\mu=0$. This contradicts to the assumption that $\|v_{h_k}\|_{L^2}=\|u_{h_k}\|_{L^2}=1+o(1)$, as $k\rightarrow\infty$. The proof of Proposition \ref{resolvent-HFsemi} is now complete.

\appendix
\renewcommand{\appendixname}{Appendix~\Alph{section}}

\section{Some commutator estimates}
\begin{lemma}\label{calculus} 
Assume that $b(x,y,\xi)\in L^{\infty}(\mathbb{R}_{x,y,\xi}^{3d})$ such that
$$ |\partial_{\xi}^{\alpha}b(x,y,\xi)|\lesssim_{\alpha} \langle\xi\rangle^{-(|\alpha|+1)}
$$
for all multi-index $\alpha\in\mathbb{N}^d, |\alpha|\leq d+1$. Then the operator $T_h$ associated with the Schwartz kernel
$$ K_h(x,y):=\frac{1}{(2\pi h)^d}\int_{\mathbb{R}^d}b(x,y,\xi)e^{\frac{i(x-y)\cdot\xi}{h}}d\xi 
$$ 
is bounded on $L^2(\mathbb{R}^d)$, uniformly in $0<h\leq 1$.
\end{lemma}

\begin{proof}
Using the Littlewood-Paley decomposition, we can decompose the operator $T_h=\sum_{j\geq 0}T_{h,j}$ where each $T_{h,j}$ has the Schwartz kernel
$$ K_{h,j}(x,y)=\frac{1}{(2\pi h)^d}\int_{\mathbb{R}^d}b_j(x,y,\xi)e^{\frac{i(x-y)\cdot\xi}{h}}d\xi,
$$ 
with $b_j(x,y,\xi)=b(x,y,\xi)\psi_j(\xi)$ and $\psi_j(\xi)=\psi(2^{-j}\xi)$ for some $\psi\in C_c^{\infty}(\frac{1}{2}\leq |\xi|\leq 2 )$, if $j\geq 1$ and $\psi_0(\xi)$ is supported on $|\xi|\leq 1$.
Note that
$$ (x-y)^{\alpha}K_{h,j}(x,y)=\frac{i^{-|\alpha|}h^{|\alpha|}}{(2\pi h)^d}\int_{\mathbb{R}^d} D_{\xi}^{\alpha}b_j(x,y,\xi) \cdot e^{\frac{i(x-y)\cdot\xi}{h}}d\xi, 
$$
we have
$$ |K_{h,j}(x,y)|\lesssim_{\alpha}\frac{h^{|\alpha|-d}}{|x-y|^{|\alpha|}}\cdot 2^{-j(|\alpha|+1-d)}.
$$
We have another trivial bound
$ |K_{h,j}(x,y)|\lesssim_{\alpha} 2^{jd}h^{-d}. 
$
Therefore, for fixed $x\in\mathbb{R}^d$, by choosing $|\alpha|=d+1$, we have
\begin{align*}
&\int_{\mathbb{R}^d}|K_{h,j}(x,y)|dy\lesssim  \int_{\mathbb{R}^d}\min\big\{2^{-j}\frac{2^{-j}h}{|x-y|^{d+1}}, 2^{jd}h^{-d}\big\} dy\\
\leq & \int_{|z|\leq 2^{-\frac{j}{d+1}}\cdot 2^{-j}h} 2^{jd}h^{-d}dz+\int_{|z|>2^{-\frac{j}{d+1}}\cdot 2^{-j}h } \frac{2^{-j}2^{-j}h }{|z|^{d+1}}dz
\lesssim  2^{-\frac{jd}{d+1}}.
\end{align*}
Similarly,  for fixed $y\in\mathbb{R}^d$, 
$$ \int_{\mathbb{R}^d}|K_{h,j}(x,y)|dx\lesssim 2^{-\frac{jd}{d+1}}.
$$
By Schur's test, we have $\|T_{h,j}\|_{\mathcal{L}(L^2)}\lesssim 2^{-\frac{jd}{d+1}}$. Using the triangle inequality, we obtain that $T_h$ is bounded on $L^2(\mathbb{R}^d)$, uniformly in $0<h\leq 1$. The proof of Lemma \ref{calculus} is now complete.
\end{proof}

\begin{coro}\label{Commutator:Lip} 
Assume that $\kappa\in W^{1,\infty}(\mathbb{R}^d)$ and $b\in S^0(\mathbb{R}^{2d})$ is a symbol of order zero, then we have
$$   \|[\mathrm{Op}_h(b),\kappa]\|_{\mathcal{L}(L^2)}=O(h). 
$$
\end{coro}

\begin{proof}
The kernel of $[\mathrm{Op}_h(b),\kappa]$ is given by
$$ K(x,y)=\frac{1}{(2\pi h)^d}\int_{\mathbb{R}^d}b(x,\xi)(\kappa(y)-\kappa(x))e^{\frac{i(x-y)\cdot\xi}{h}}d\xi.
$$
Since $\kappa\in W^{1,\infty}$, there exists $\Psi\in L^{\infty}(\mathbb{R}^d;\mathbb{R}^d)$ such that
$$ \kappa(y)-\kappa(x)=(y-x)\cdot \Psi(x,y).
$$
Thus
$$ K(x,y)=-\sum_{j=1}^d\frac{h}{i(2\pi h)^d}\int_{\mathbb{R}^d}\partial_{\xi_j}b(x,\xi)\Psi_j(x,y))e^{\frac{i(x-y)\xi}{h}}d\xi
$$
Applying Lemma \ref{calculus} to each $\partial_{\xi_j}b(x,\xi) \Psi_j(x,y)$, the proof of Corollary \ref{Commutator:Lip} is complete.
\end{proof}


\end{document}